\documentclass[12pt]{amsart}
\usepackage{graphicx}
\usepackage{verbatim}   
\usepackage{xcolor,colortbl}
\usepackage{url}
\usepackage{appendix}

\newtheorem{theorem}{Theorem}
\newtheorem{lemma}[theorem]{Lemma}

\def\Qnum{{\mathbb{Q}}} 
\def\Znum{{\mathbb{Z}}} 
\newcommand{\br}[1]{\langle {#1} \rangle }

\def\cO{{\mathcal O}}

\def\ord{\mbox{\rm{ord}}}

\title{On Difference Sets with small $\lambda$}

\dedicatory{Dedicated to K.T. Arasu on the occasion of his 65th birthday.}

\author{Daniel M. Gordon}

\address{IDA Center for Communications Research\\
4320 Westerra Court\\
San Diego, CA 92121\\
USA}
\email{gordon@ccrwest.org}

\begin{document}

\maketitle

\begin{abstract}
In a 1989 paper \cite{arasu2}, Arasu used an observation about multipliers
to show that no $(352,27,2)$ difference set exists
in any abelian group.  The proof is quite short and required no computer
assistance.  We show that it may be applied to a wide range of
parameters $(v,k,\lambda)$, particularly for small values of $\lambda$.
With it a computer search was able to show that the Prime Power
Conjecture is true up to order $2 \cdot 10^{10}$, extend Hughes and
Dickey's computations for $\lambda=2$ and $k \leq 5000$ up to
$10^{10}$, and show nonexistence for many other parameters.
\end{abstract}

\section{Introduction}

A $(v,k,\lambda)$-difference set $D$ in a group $G$ of order $v$ is a set
$\{d_1,d_2,\ldots,d_k\}$ of elements from $G$ such that every nonzero
element of $G$ has exactly $\lambda$ representations as $d_i - d_j$.
The {\em order} of $D$ is $n=k-\lambda$.


A {\em (numerical) multiplier} is an integer $m$ for which 
multiplication of each $d_i$ by $m$ produces a shift of the original
difference set: $mD = D+g$ 
for some $g\in G$.  The
set of multipliers form a group $M$, and it is well-known that some
translate of $D$ is fixed by $M$.
This implies that a shift of $D$ can be written as a union of orbits of $G$ under $M$.

The First Multiplier Theorem states that any prime $p>\lambda$ which
divides $n$ and not $v$ must be a multiplier of $D$.  The Multiplier
Conjecture is that the $p>\lambda$ condition is not needed.  This is
still open, but there have been many strengthenings of the First
Multiplier Theorem; see \cite{gs} for recent results.

Many difference set parameters can be dealt with by finding a group of
multipliers $M$ and looking at the resulting
orbits.  For instance, it may be that no union of orbits has
size $k$, or the set of orbits may be small enough that all
possibilities may be checked with a short search.
Lander, in \cite{lander}, gives many such examples.

Arasu \cite{arasu2} 
showed that no abelian biplanes (difference sets with $\lambda=2$)
of order 25 exist.  Our main tool will be a generalization of his
argument, which we restate here.

\begin{theorem}\label{thm352}
  No $(352,27,2)$ difference set exists in any abelian group $G$.
\end{theorem}

\begin{proof}
Any such difference set has $5$ as a multiplier.  Take $p=11$, and $H$
a group of order 32 so that $G = \Znum_{11} \times H$.
Then
$5^{8} \equiv 1 \pmod {32}$, and so fixes $H$.
The orbits of $\br{5^{8}}$ in $\Znum_{11}$ are $\{0\}$,
$\{1,3,4,5,9\}$, and $\{2,6,7,8,10\}$.  The orbits in $G$ are just
these orbits with a fixed element $h \in H$.

A difference set $D$ made up of these orbits will have a certain
number $a$ of 5-orbits $\br{(1,h)}$ and $\br{(2,h)}$, and 
$b=27-5a$ 1-orbits.  There are $b (b-1)$ differences of the
singleton orbits, each of which is of the form $(0,h)$ with $h \neq 0$.  There are
$31$ such elements, and each must occur exactly twice as a
difference of elements of $D$, and so $b (b-1) \leq 31 \cdot 2 =
62$.

This means that we must have $b<9$, and so $a \geq 4$.  But the
20 differences from elements in one 
5-orbit are all of the form $(x,0)$, $x \neq 0$.  There are 10
such elements, and in fact each of them occurs exactly twice in the
differences of one 5-orbit.  Since we have multiple 5-orbits, these
elements will occur as differences too many times.
\end{proof}

One nice feature of this argument is that it takes care of all abelian
groups $G$ of order 352
at once.  Other arguments (\cite{adjp}, \cite{lander}) only handle
specific groups.

\section{Extending the Method}

It is clear that Arasu's method can be applied to other parameter
sets.  In this section we give a generalization of Theorem~\ref{thm352}.

  \begin{lemma}
  Let $G = \Znum_p \times H$, where $H$ is abelian and $\gcd(p,|H|)=1$.  Let $m$ be a
  multiplier of a $(v,k,\lambda)$ difference set, and $s$ be the
  smallest positive integer for which $m^s \equiv 1
  \pmod{\exp(H)}$.
Then the orbits of $G$ under $\br{m^s}$ are of the form $(\cO,h)$,
for fixed $h \in H$.
There are exactly $|H|$ orbits $(0,h)$ of size 1, and the remaining orbits all have
the same size $o = \ord_p(m^s)$.
  \end{lemma}

\begin{proof}
  The proof of this is the same as for Theorem~\ref{thm352}.  
The group of multipliers generated by $m^s$ will fix all $h \in H$
Because $p$ is prime, all the nonzero orbits of $\Znum_p$ under this group
will have the same size, some divisor of $p-1$.
\end{proof}

Now for any $(v,k,\lambda)$, if we can find a prime
$p|v$ and multiplier $m$ for which $m^s$ has a reasonably large order
mod $p$, we can look at differences of the 1-orbits and $o$-orbits and
try to get a contradiction: if there are $a$ orbits of size $o$, and
$b$ 1-orbits, then we have:

\begin{theorem}\label{thm:zph}
  Let $G = \Znum_p \times H$, where $H$ is abelian and $\gcd(p,|H|)=1$.  Let $m$ be a
  multiplier of a $(v,k,\lambda)$ difference set, and $s$ be the
  smallest positive integer for which $m^s \equiv 1
  \pmod{\exp(H)}$, and $o = \ord_p(m^s)$.  If there is no solution
  in positive integers $a$ and $b$ to:
\begin{eqnarray}
  k & = & a o + b \label{eq:k} \\
b(b-1)  & \leq & \lambda  (|H|-1) \label{eq:b} \\
a \cdot o(o-1)  & \leq & \lambda (p-1) \label{eq:a}
\end{eqnarray}
then no $(v,k,\lambda)$ difference set exists in $G$.
\end{theorem}

This method will be most useful when $\lambda$ is small, since each
element can only occur $\lambda$ times as a difference, so whatever
the choice of orbits either elements of the form $(x,0)$ or $(0,h)$
are likely to occur too many times.  Still, when $n$ and $v$ have
large prime factors ($n$ so that we have a known multiplier, and $v$
so that we have a suitable $p$ to use in Theorem~\ref{thm:zph}), it
can still often be applied.  


When Theorem~\ref{thm:zph} fails, if $G$ is cyclic we will sometimes use the theorem of
Xiang and Chen \cite{xc95}:

\begin{theorem}\label{thm:multgp}
  Let $D$ be a $(v,k,\lambda)$ difference set in a cyclic group $G$
  with multiplier group $M$. Except for the $(21,5,1)$ difference set,
  $|M|\leq k$.
\end{theorem}


This theorem may be extended to contracted multipliers as well 
(see Section~VI.5 of
\cite{bjl} for information about difference lists and contracted
multipliers).  

\begin{theorem}\label{thm:gh}
  Let $D$ be a $(v,k,\lambda)$ difference set in a cyclic group $G$, 
and $H$ be the subgroup of $G$ of order $h$ and index $u$.
Then with the same exception, the group
  $M$ of $G/H$-multipliers has order $|M| \leq k$.
\end{theorem}

\begin{proof}
  The proof is exactly the same as the proof of
  Theorem~\ref{thm:multgp} in \cite{xc95}, replacing multipliers with contracted
  multipliers.  
$M$ is isomorphic to a subgroup of ${\rm Gal}\  \Qnum(\zeta_u)/\Qnum$,
where $\zeta_u$ is a primitive $u$th root of unity.
Let 
$$S = \overline{D} = \{ \overline{d_1},\overline{d_2},\ldots,
  \overline{d_k} \}$$ 
be the $(u,k,h,\lambda)$ difference list over $G/H$ obtained by
  sending the elements of $D$ to their image in $G/H$.  By Theorem
  5.14 of \cite{bjl}, we may assume that $S$ is fixed by $M$. Let
$\chi$ be a generator of the character group of
$G/H$, 
$K = \Qnum\left( \chi(S),\chi^2(S),\ldots,\chi^{u-1}(S)\right),$
and $\alpha_t$ be the field automorphism sending 
$\zeta_u \mapsto \zeta_u^t$.  As in \cite{xc95}, we may show
that ${\rm Gal}\  \Qnum(\zeta_u)/K = M$. If $t \in M$ it fixes $S$,
so $\alpha_t$ fixes $\chi(S)$. If $\alpha_t$ fixes $\chi^i(S)$ for
$i=1,2,\ldots,u-1$, then by Fourier inversion $t$ fixes $S$, and so is
in $M$.

Now let
$$f(X) = \prod_{i=1}^k  \left( X - \chi(\overline{d_i})\right).$$
The coefficients of $f(X)$ are elementary symmetric polynomials in the 
$\chi(\overline{d_i})$, which are fixed by $\alpha_t$
for any $t \in M$, so 
$f(X) \in K[X]$.

By Theorem 1 of Cohen~\cite{scohen}, if $D$ is not the (21,5,1)
difference set, then at least one of the $d_i$ is relatively prime to
$v$, and so $\chi(\overline{d_i})$ is a primitive $u$th root of unity.
It is also a root of $f(X)$, and so
$$|M| = \left[ \Qnum(\zeta_u):K\right] \leq \deg f(X) = k.$$
\end{proof}

\section{The Prime Power Conjecture}

A $(v,k,1)$ difference set is called a {planar abelian difference set}.
These exist if $n=k-1$
is a prime power, and the Prime Power Conjecture (PPC) is that these are the
only ones.  In \cite{gordon} it was shown that the PPC
is true for all groups for orders up to $2\cdot 10^6$, and in \cite{bg} for cyclic
groups for orders up to $2\cdot 10^9$.  Peluse \cite{2003.04929}
recently showed that the PPC is asymptotically true; the number of 
orders up to $N$ for which planar abelian difference sets exist is $O(N/\log N)$,
the same as the number of prime powers.

In these papers non-prime power orders were eliminated by a series of tests; see
\cite{gordon} for details.  The initial tests only depended on the
prime factors of $n$ and $v$, and were very fast.  
Tables 1 and 2 in \cite{gordon} gave lists of $(v,k,1)$ planar abelian
difference set parameters which could not be eliminated with these tests.
To show they did not exist, Proposition 5.11 of Lander \cite{lander} was
used:
\begin{theorem}\label{thm:511}
If $t_1, t_2, t_3, t_4$ are numerical multipliers of a $(v,k,1)$
difference set in $G$, and 
$$t_1-t_2 \equiv t_3-t_4 \pmod  {\exp (G)},$$
then 
$\exp(G)$ divides ${\rm lcm}(t_1-t_2,t_1-t_3)$.
\end{theorem}

For each case a large number of multipliers were generated, until either a prime known
not to be an extraneous multiplier was discovered, or two pairs of
multipliers with the same difference modulo $\exp(G)$ were found, so
that Theorem~\ref{thm:511} could be applied.
These calculations required a substantial amount of computation time
and memory.

With Theorem~\ref{thm:zph} the hard cases from \cite{gordon} can be eliminated
quickly. 
To illustrate the power of the theorem, Table~\ref{tab:thm3} gives
parameters used in Theorem~\ref{thm:zph} to eliminate some of the
parameters in the tables in \cite{gordon};
with the value of $o$ in
the last column, it is easy to check that there are no positive
integers $a$ and $b$ solving equations (\ref{eq:k}), (\ref{eq:b}) and (\ref{eq:a}).

\begin{table}
\begin{center}
\begin{tabular}{|c|c|c|c|c|}
\hline
$k$ & $p$ & $|H|$  & $m^s$ & $\ord_p(m^s)$ \\\hline
$2436$ & $5931661$ & $1$ & $5^1$ & $435$ \\
$24452$ & $199291951$ & $3$ & $499^1$ & $6175$ \\
$45152$ & $22651$ & $90003$ & $277^{789}$ & $25$ \\
$56408$ & $24781$ & $128397$ & $4339^{63}$ & $295$ \\
$58724$ & $450601$ & $7653$ & $8389^{75}$ & $751$ \\ \hline
$2444$ & $109$ & $54777$ & $7^{465}$ & $9$ \\
$3234$ & $4759$ & $2197$ & $61^{507}$ & $61$ \\
$72012$ & $35911$ & $144403$ & $673^{245}$ & $513$ \\
$73482$ & $149113$ & $36211$ & $373^{9}$ & $2071$ \\
\hline
\end{tabular}
\caption{Small $(v,k,1)$ parameters from Tables 1 and 2 of \cite{gordon} eliminated by Theorem~\ref{thm:zph}}
\label{tab:thm3}
\end{center}
\end{table}

Using
Arasu's method allows the computations to be redone in a different
manner.  In addition, it requires far less work for the hard cases, so
it was possible to take the computations further.  
Replicating the search up to $2\cdot 10^6$ took under a minute on a workstation.
A longer run using the fast tests from \cite{gordon} and
Theorem~\ref{thm:zph}
eliminated every order up to $2 \cdot 10^{10}$
except for the ones given in Table~\ref{tab:ppc},
which were then eliminated using Theorem~\ref{thm:511}.
Note that the first two values of $k$ were missing from the tables in
\cite{gordon}.


\begin{table}
\begin{center}
\begin{tabular}{|c|c|c|}
\hline
$k$ & $n$ & $v$ \\\hline
$1096386$ & $5 \cdot 219277$ & $79 \cdot 109 \cdot 1951 \cdot 71551$ \\
$1320794$ & $373 \cdot 3541$ & $3 \cdot 11551 \cdot 50341831$ \\
$2378196$ & $5 \cdot 475639$ & $211 \cdot 631 \cdot 3319 \cdot 12799$ \\
$20846324$ & $61 \cdot 341743$ & $3 \cdot 88951 \cdot 1628496601$ \\
$40027524$ & $107 \cdot 374089$ & $7 \cdot 13\cdot 3541 \cdot 54163 \cdot 91801$ \\ 
$2830957656$ & $5 \cdot 566191531$ & $109^{2} \cdot 1171 \cdot 1231 \cdot 1951 \cdot 239851$\\
$7700562788$ &  $9817 \cdot 784411$ & $3 \cdot 61^{2} \cdot 1831 \cdot 1703287^{2}$\\
\hline
\end{tabular}
\caption{$(v,k,1)$ parameters up to $k=2 \cdot 10^{10}$ not eliminated by Theorem~\ref{thm:zph}}
\label{tab:ppc}
\end{center}
\end{table}

Unlike the fast tests in \cite{gordon}, for which the number passing
was roughly linear in the bound on $n$, Theorem~\ref{thm:zph} gets more
effective for larger orders, since it becomes increasingly likely that
$v$ will have a large prime factor $p$ for which some prime divisor of
$n$ has large order mod $p$.  All values of $k$ between $7.7 \cdot
10^9$ and $2 \cdot 10^{10}$ were eliminated, and a heuristic argument
suggests that the number of cases up to order $n$ passing
Theorem~\ref{thm:zph} will be at most $O(\log n)$.

\section{Biplanes}

Theorem~\ref{thm352} was also shown by Hughes in \cite{hughes}.
Computations by Hughes and Dickey reported in that paper showed that no
abelian $(v,k,2)$ difference sets exist with order less than 5000,
except for the known cases $k=3,4,5, 6$ and $9$.  They give few details
about their method;
it is possible that their method was something similar to
that of Arasu.


\begin{table}
\begin{center}
\begin{tabular}{|c|c|c|}
\hline
$k$ & $n$ & $v$ \\\hline
$47433$	&	$47431$	&	$13693 \cdot 82153$\\
$86013$	&	$86011$	&	$7 \cdot 71 \cdot 883 \cdot 8429$\\
$890196$ & $2 \cdot 445097$ & $396224014111$\\
$1120521$	&	$1120519$	&	$83059 \cdot 7558279$\\
$1767189$	&	$1767187$	&	$7 \cdot 223068228181$\\
$937097469$	&	$937097467$	&	$19942759 \cdot 22016804833$\\
\hline
\end{tabular}
\caption{Open $(v,k,2)$ cases for $k \leq 10^{10}$}
\label{tab:biplane}
\end{center}
\end{table}

A run up to order 
$10^{10}$
eliminated all but $24$ parameters.
Most of the rest were dealt with using
Theorems 4.19 and 4.38 of Lander \cite{lander}.
Table~\ref{tab:biplane} gives the remaining open cases.

Theorem~\ref{thm:gh} was an important tool for eliminating open cases in this
and the next table.
Biplanes of order a power of $4$, such as $(525826,1026,2)$, pass
Theorem~\ref{thm:zph}, and have no known multipliers, so the standard
methods are no help. However, in each case up to order $2^{30}$
we have that $G$ is cyclic, 2 is a $G/H$ multiplier for $H$ the
group of order 2 by the Contracted Multiplier Theorem (Corollary 5.13
of \cite{bjl}), 
and the order $\ord_{v/2}(2)$ is larger than $k$, showing
that those biplanes do not exist.

\section{General Parameters}

Theorem~\ref{thm:zph} may be applied for larger $\lambda$; while more
parameters will slip through because of a lack of known multipliers or
Equations (\ref{eq:b}) and (\ref{eq:a}) being less restrictive, many
may still be eliminated.  A run was done for difference sets with
$\lambda=3$ up to order $10^{10}$.  There were 269 parameters that
passed Theorem~\ref{thm:zph}, but most were then eliminated with
Theorems~\ref{thm:multgp} and \ref{thm:gh}, the Lander tests, and the
Mann test (\cite{bjl}, Theorem VI.6.2).
Table~\ref{tab:triplane} shows the six remaining cases.

\begin{table}
\begin{center}
\begin{tabular}{|c|c|c|}
\hline
$k$ & $n$ & $v$ \\\hline
$120$	&	$3^{2} \cdot 13$	&	$3^{2} \cdot 23^{2}$\\
$441$	&	$2 \cdot 3 \cdot 73$	&	$71 \cdot 911$\\
$2350$	&	$2347$	&	$1840051$\\
$740406$	&	$3^{2} \cdot 82267$	&	$3^{4} \cdot 19391 \cdot 116341$\\
$3793567$	&	$2^{2} \cdot 948391$	&	$5^{2} \cdot 251 \cdot 397 \cdot 463 \cdot 4159$\\
$289842739$	&	$2^{4} \cdot 18115171$	&	$3 \cdot 5 \cdot 23 \cdot 103^{2} \cdot 137 \cdot 223^{2} \cdot 1123$\\
\hline
\end{tabular}
\caption{Open $(v,k,3)$ cases for $k \leq 10^{10}$}
\label{tab:triplane}
\end{center}
\end{table}


The author has set up the La Jolla Difference Set Repository \cite{ljdsr}, an
online database containing existence results for parameters up to
$v=10^6$, as well as a large number of known difference sets.
There are $1.44$ million parameters that pass basic counting and the
BRC theorem, of which about $180{,}000$ were open.
Applying Theorems~\ref{thm:zph} and \ref{thm:gh}
resolved over $50{,}000$ of them.

\section*{Acknowledgement}

We thank the anonymous referee for suggestions that
led to Theorem~\ref{thm:gh}.

\bibliographystyle{plain}

\begin{thebibliography}{10}

\bibitem{arasu2}
K.~T. Arasu.
\newblock Singer groups of biplanes of order 25.
\newblock {\em Arch. Math.}, 53:622--624, 1989.

\bibitem{adjp}
K.T. Arasu, J.~Davis, D.~Jungnickel, and A.~Pott.
\newblock A note on intersection numbers of difference sets.
\newblock {\em Europ. J. Comb.}, 11:95--98, 1990.

\bibitem{bg}
L.~D. Baumert and D.~M. Gordon.
\newblock On the existence of cyclic difference sets with small parameters.
\newblock In Van~Der Poorten and Stein, editors, {\em Conference in Number
  Theory in Honour of Professor H.C. Williams}, pages 61--68, 2004.

\bibitem{bjl}
T.~Beth, D.~Jungnickel, and H.~Lenz.
\newblock {\em Design Theory}, volume~1 of {\em Encyclopedia of Mathematics and
  its Applications}.
\newblock Cambridge University Press, 2 edition, 1999.

\bibitem{scohen}
Stephen~D. Cohen.
\newblock Generators in cyclic difference sets.
\newblock {\em JCT A}, 51:227--236, 1989.

\bibitem{gordon}
D.~M. Gordon.
\newblock The prime power conjecture is true for $n<2{,}000{,}000$.
\newblock {\em Electronic J. Combinatorics}, 1, 1994.
\newblock R6.

\bibitem{ljdsr}
D.~M. Gordon.
\newblock La {J}olla {D}ifference {S}et {R}epository.
\newblock \url{https://www.dmgordon.org/diffsets}, 2020.

\bibitem{gs}
D.~M. Gordon and B.~Schmidt.
\newblock On the multiplier conjecture.
\newblock {\em Designs, Codes and Crypt.}, pages 221--236, 2016.

\bibitem{hughes}
D.~Hughes.
\newblock Biplanes and semi-biplanes.
\newblock In D.~A. Holton and Jennifer Seberry, editors, {\em Combinatorial
  Mathematics}, pages 55--58. Springer Berlin Heidelberg, 1978.

\bibitem{lander}
E.~S. Lander.
\newblock {\em Symmetric Designs: An Algebraic Approach}, volume~74 of {\em LMS
  Lecture Note Series}.
\newblock Cambridge, 1983.

\bibitem{2003.04929}
Sarah Peluse.
\newblock An asymptotic version of the prime power conjecture for perfect
  difference sets, 2020.

\bibitem{xc95}
Q.~Xiang and Y.Q. Chen.
\newblock On the size of the multiplier groups of cyclic difference sets.
\newblock {\em JCT A}, 69:168--169, 1995.

\end{thebibliography}

\end{document}